\documentclass{actams-author}

\numberwithin{equation}{section} %%% Equations numbered by section. If you don't want it, please delete it.

\begin{document}
\ID{E15-523} % The manuscript number Exx-xxx is filled in when your paper is numbered by editor
 \DATE{Final, 2013-xx-xx} % date completed for type of submission£ºDraft,Revision(The first revision,the second revision...),Final,Proof
 \PageNum{1}
 \Volume{2017}{}{37}{1} %{Year},{},{Vol.},and{No.} are filled in when your paper are accepted for publication,such as {2012}{}{32}{4}
 \EditorNote{$^*$Received October 18, 2015; revised April 18, 2016.}
 \EditorNote{$^\dag$Corresponding author: Mohammed D. KASSIM.} % for example,Received May 24,2012;revised October 17,2012.The Fund which you are supported by

\abovedisplayskip 6pt plus 2pt minus 2pt \belowdisplayskip 6pt
plus 2pt minus 2pt
%%%%%%%%%%%%%%%%
\def\vsp{\vspace{1mm}}
\def\th#1{\vspace{1mm}\noindent{\bf #1}\quad}
\def\proof{\vspace{1mm}\indent{\bf Proof}\quad}
\def\no{\nonumber}
\newtheorem{conclusion}{Conclusion}[section]
\newenvironment{prof}[1][Proof]{\indent\textbf{#1}\quad }
{\hfill $\Box$\vspace{0.7mm}}
\def\q{\quad} \def\qq{\qquad}
\allowdisplaybreaks[4]
%%%%%%%%%%%%%%%%%%%%%%%%%%%%%%%%%%%%%%%%%%%%%%%%%%%%%%%%%%%%%%%%%%%%%%%%%%%%%%%%%%%%%%%%%%%%%%%
%%-------------------       Beginning of  Author's Definitions       -------------------%%
%%                     Note: You may add your own definitions here.

%%-------------------         the end of  Author's Definitions           -------------------%%

\AuthorMark{M. Kassim, K. Furati \& N.-e. Tatar}                             %%%  appear on the head of even pages  %%%

\TitleMark{\uppercase{Non-existence for fractionally damped fdps}}  %%% Running Title,the abbreviated title not exceeding 45 characters(includeing spaces), appear on the head of odd pages  %%%

\title{\uppercase{ Non-existence for fractionally damped fractional differential problems }        %%%   Main Title of your paper  %%%
\footnote{ }}                 %%%  $^\dag$Corresponding author   %%%
\author{\sl{Mohammed \uppercase{D. Kassim $^\dag$}}}    %%%  1st Author's information  and use initial capitals %%%
   { King Fahd University of Petroleum and Minerals, Department of Mathematics and Statistics, Dhahran, 31261, Saudi Arabia\\
    E-mail\,$:$ dahan@kfupm.edu.sa}

\author{\sl{Khaled \uppercase{m. Furati}}}    %%%  2nd Author's info, if exists, or you may delete this part directly ;$^{\dag}$ mark Corresponding author %%%
{ King Fahd University of Petroleum and Minerals, Department of Mathematics and Statistics, Dhahran, 31261, Saudi Arabia\\
    E-mail\,$:$ kmfurati@kfupm.edu.sa }  %The left is the address information of the author Lungang Liu of this paper

\author{\sl{Nasser-eddine \uppercase{Tatar}}}    %%%  1st Author's information  and use initial capitals %%%
   { King Fahd University of Petroleum and Minerals, Department of Mathematics and Statistics, Dhahran, 31261, Saudi Arabia\\
    E-mail\,$:$ tatarn@kfupm.edu.sa}

\maketitle%

\Abstract{In this paper, we are concerned with a fractional
differential inequality containing a lower order fractional
derivative and a polynomial source term in the right hand side. A
non-existence of non-trivial global solutions result is proved in
an appropriate space by means of the test-function method. The
range of blow up is found to depend only on the lower order
derivative. This is in line with the well-known fact for an
internally weakly damped wave equation that solutions will
converge to solutions of the
parabolic part.}      % the abstract

        % the keywords

\MRSubClass{26A33}      % 2010 MR Subject Classification

\section{Introduction}

In this paper we consider the problem%
\begin{equation}
\left\{
\begin{array}{l}
D_{0}^{\alpha }y\left( t\right) +D_{0}^{\beta }y\left( t\right)
=f\left(
t,y\left( t\right) \right) ,\text{ }t>0, \\
I_{0}^{1-\alpha }y\left( t\right) |_{t=0}=b,%
\end{array}%
\right.  \label{1}
\end{equation}%
where $D_{0}^{\sigma }$ is the Riemann-Liouville fractional
derivative of order $\sigma >0,$ $0<\beta \leq \alpha \leq
1.$\newline
A nonexistence result of non-trivial global solutions for the problem (\ref%
{1}) will be proved when $f\left( t,y\left( t\right) \right) \geq
t^{\gamma }\left\vert y\left( t\right) \right\vert ^{m}$ for some
$m>1$ and $\gamma
\in \mathbf{%
%TCIMACRO{\U{211d} }%
%BeginExpansion
\mathbb{R}
%EndExpansion
.}$ That is we consider the problem:%
\begin{equation}
\left\{
\begin{array}{l}
D_{0}^{\alpha }y\left( t\right) +D_{0}^{\beta }y\left( t\right)
\geq
t^{\gamma }\left\vert y\left( t\right) \right\vert ^{m},\text{ }t>0, \\
I_{0}^{1-\alpha }y\left( t\right) |_{t=0}=b,%
\end{array}%
\right.  \label{2}
\end{equation}%
where $0<\beta \leq \alpha \leq 1$ and show that no solutions can
exist for all time for certain values of $\gamma $ and $m$. In
particular, we find the range of values of $m$ for which solutions
do not exist globally. Clearly, sufficient conditions for
nonexistence provide necessary conditions for existence of
solutions.\newline The interest to fractional calculus has been
accelerated the past three
decades after the publication of the three papers of Bagley and Torvik \cite%
{Bagley1, Bagley2, Bagley3} and the paper by Podlubny
\cite{Podl(2002)}. Many phenomena in diverse fields of science and
engineering can be described by differential equations of
non-integer order. Namely, they arise naturally in
viscoelasticity, porous media, electrochemistry, control and
electromagnetic, etc \cite{Miller-Ross(1993), Oldham-Spanier, Podl(1999)}.%
\newline
In fact it has been shown by experiments that derivatives of
non-integer order can describe many phenomena better than
derivatives of integer order specially hereditary phenomena and
processes.\newline Some recent applications arose in
viscoelasticity, rheology, control
systems, synthesis, robots and nanotechnology, etc (see \cite%
{Hilfer,Kilbas,V,R,M1,M2,P}).\newline Regarding the existence of
solutions for various classes of fractional
differential equations, there are many results (e.g. see \cite%
{A1,A2,F1,F2,F3,Ka2,M,T}). For the issue of nonexistence of
solutions for
fractional differential equations, we refer to \cite%
{F4,Ka1,Laskri-Tatar(2010),Tatar} and to \cite{K1,K2,K3,K4} for
partial differential equations involving fractional derivatives
(see also references therein).\newline The existence and
uniqueness of solutions for problem (\ref{1}) has been discussed
in \cite{Kilbas}.\newline In case $\alpha =\beta =1$\ and $f\left(
t,y\left( t\right) \right)
=2y^{m}\left( t\right) $\ in (\ref{1}) we obtain%
\begin{equation*}
\left\{
\begin{array}{l}
y^{\prime }\left( t\right) =y^{m}\left( t\right) , \\
y\left( t\right) |_{t=0}=b.%
\end{array}%
\right.
\end{equation*}%
This problem has, for $m>1$, the solution\bigskip
\begin{equation*}
y(t)=\left[ \left( 1-m\right) \left( t+c\right) \right] ^{1/\left(
1-m\right) },
\end{equation*}%
where%
\begin{equation*}
c=\frac{b^{1-m}}{1-m}.
\end{equation*}%
Observe that, for $m>1,$ the solution blows-up in finite
time.\newline When $\alpha =1,$ $\beta =0$ and $\gamma =0$, the
problem (\ref{2}) with an equality instead of inequality is
equivalent to the Bernoulli differential
problem%
\begin{equation}
\left\{
\begin{array}{l}
y^{\prime }\left( t\right) +y\left( t\right) =y^{m}\left( t\right) ,\text{ }%
t>0, \\
y\left( t\right) |_{t=0}=b.%
\end{array}%
\right.  \label{Bernoulli}
\end{equation}%
\newline
The solution of (\ref{Bernoulli}) is given by%
\begin{equation*}
y\left( t\right) =\left[ 1+\left( b^{1-m}-1\right) \exp \left( m-1\right) t%
\right] ^{1/\left( 1-m\right) }.
\end{equation*}%
Clearly $y(t)$ blows up in the finite time%
\begin{equation*}
c=\frac{1}{1-m}\ln \left( 1-b^{1-m}\right) ,\text{ }m,\text{ }b>1.
\end{equation*}%
In case $\alpha =\beta $\ in (\ref{2}) we obtain the problem with
only one
fractional derivative%
\begin{equation}
\left\{
\begin{array}{l}
2D_{0}^{\alpha }y\left( t\right) \geq t^{\gamma }\left\vert
y\left( t\right)
\right\vert ^{m},\text{ }t>0, \\
I_{0}^{1-\alpha }y\left( t\right) |_{t=0}=b.%
\end{array}%
\right.  \label{3}
\end{equation}%
Problem (\ref{3}) has been considered by Laskri and Tatar \cite%
{Laskri-Tatar(2010)}. It was shown that if $\gamma >-\alpha $\textit{\ }and%
\textit{\ }$1<m\leq \frac{\gamma +1}{1-\alpha }$, then, Problem
(\ref{3}) does not admit global nontrivial solutions\
when\textit{\ }$b\geq 0.$\newline Here, we would like to
investigate the case where a lower order fractional derivative is
present in the equation (or inequality). It is known that for
hyperbolic equations, say the wave equation with an internal
fractional damping represented by the first derivative (i.e.
$\alpha =2,$ $\beta =1$ also known as the Telegraph equation),
this damping has a dissipation effect. It will compete with the
polynomial source and may take it over this blowing-up term under
certain circumstances. Moreover, it has been shown for the
telegraph problem that solutions approach the solution of the same
problem without the highest derivative when $t$ goes to infinity
(that is the parabolic equation). This result has been generalized
to the fractional derivative case in \cite{Cascaval} and in
\cite{Tatar}.\newline
For our problem here (\ref{2}), we would like to see how much influential $%
D_{0}^{\beta }y$ will be on the blow-up phenomenon. In particular,
how the range of values $m$ ensuring blow-up in finite time would
be affected. We reached the conclusion that here also it is the
lower order derivative (i.e. $\beta $) which determines the range
of blow-up just like the parabolic part in the hyperbolic
problem.\newline The rest of the paper is divided into two
sections. In Section \ref{pre}, we present some definitions,
notations, and lemmas which will be needed later in our proof.
Section \ref{non-ex} is devoted to the nonexistence result.

\section{Preliminaries\label{pre}\protect\bigskip}

In this section we present some definitions, lemmas, properties
and notation which will be used in our result later.\newline

\begin{definition}
\textbf{\label{def:1}}\textit{The Riemann-Liouville left-sided
fractional integral }$I_{a}^{\alpha }f$\textit{\ of order }$\alpha
>0$\textit{\ is
defined by}%
\begin{equation}
I_{a}^{\alpha }f(t):=\frac{1}{\Gamma (\alpha )}\int\nolimits_{a}^{t}\frac{%
f(s)}{(t-s)^{1-\alpha }}ds,\text{ }t>a,\;\alpha >0,  \label{ri1}
\end{equation}%
\textit{provided that the integral exists}. Here $\Gamma (\alpha
)$ is the Gamma function. When $\alpha =0$, we define $I_{a}^{0}f$
$=f$.
\end{definition}

\begin{definition}
\label{def:2}\textit{The Riemann-Liouville right-sided fractional integral }$%
I_{b-}^{\alpha }f$\textit{\ of order }$\alpha >0$\textit{\ is defined by}%
\begin{equation}
I_{b-}^{\alpha }f(t):=\frac{1}{\Gamma (\alpha )}\int\nolimits_{t}^{b}\frac{%
f(s)}{(s-t)^{1-\alpha }}ds,\text{ }t<b,\;\alpha >0,  \label{ri2}
\end{equation}%
\textit{provided that the integral exists. When }$\alpha
=0$\textit{, we define }$I_{b-}^{0}f$\textit{\ }$=f$\textit{.}
\end{definition}

\begin{definition}
\label{def:3}\textit{The Riemann-Liouville left-sided fractional derivative }%
$D_{a}^{\alpha }f$\textit{\ of order }$\alpha $\textit{, }$0<\alpha <1,$%
\textit{\ is defined by\bigskip }%
\begin{equation*}
D_{a}^{\alpha }f\left( t\right) =\frac{d}{dt}I_{a}^{1-\alpha
}f(t),
\end{equation*}%
\textit{that is,}%
\begin{equation}
D_{a}^{\alpha }f\left( t\right) =\frac{1}{\Gamma (1-\alpha )}\frac{d}{dt}%
\int\nolimits_{a}^{t}\frac{f(s)}{(t-s)^{\alpha
}}ds,\;t>a,\;0<\alpha <1, \label{rd1}
\end{equation}%
\textit{when }$\alpha =1$\textit{\ we have }$D_{a}^{\alpha }f$\textit{\ }$%
=Df $\textit{. In particular, when }$\alpha =0$\textit{, }$D_{a}^{0}f$%
\textit{\ }$=f$\textit{.}
\end{definition}

\begin{definition}
\label{def:4}\textit{The Riemann-Liouville right-sided fractional
derivative
}$D_{b-}^{\alpha }f$\textit{\ of order }$\alpha ,$\textit{\ }$0<\alpha <1$%
\textit{, is defined by\bigskip }%
\begin{equation*}
D_{b-}^{\alpha }f\left( t\right) =-\frac{d}{dt}I_{b-}^{1-\alpha
}f(t),
\end{equation*}%
\textit{that is,}%
\begin{equation}
D_{b-}^{\alpha }f\left( t\right) =-\frac{1}{\Gamma (1-\alpha )}\frac{d}{dt}%
\int\nolimits_{t}^{b}\frac{f(s)}{(s-t)^{\alpha
}}ds,\;t<b,\;0<\alpha <1. \label{rd2}
\end{equation}%
\textit{In particular, when }$\alpha =0$\textit{, }$D_{b-}^{\alpha }f$%
\textit{\ }$=f$\textit{.}
\end{definition}

\begin{lemma}
\label{lem:5}\textit{(Fractional Integration by Parts)\ Let }$\alpha >0,$%
\textit{\ }$p\geq 1,$\textit{\ }$q\geq 1$\textit{\ and\ }$\frac{1}{p}+\frac{1%
}{q}\leq 1+\alpha $\textit{\ (}$p\neq 1$\textit{\ and }$q\neq
1$\textit{\ in the case when }$\frac{1}{p}+\frac{1}{q}=1+\alpha
$\textit{). If }$\varphi \in L_{p}\left( a,b\right) $\textit{\ and
}$\psi \in L_{q}\left( a,b\right)
, $\textit{\ then }%
\begin{equation}
\int\nolimits_{a}^{b}\varphi \left( t\right) \left( I_{a}^{\alpha
}\psi \right) \left( t\right) dt=\int\nolimits_{a}^{b}\psi \left(
t\right) \left( I_{b-}^{\alpha }\varphi \right) \left( t\right)
dt.  \label{fibp}
\end{equation}
\end{lemma}

\begin{definition}
\label{def:6}\textit{We consider the weighted spaces of continuous functions}%
\begin{equation*}
C_{\gamma }\left[ a,b\right] =\left\{ f:(a,b]\rightarrow \mathbb{R}\mathbf{:}%
\left( t-a\right) ^{\gamma }f\left( t\right) \in C\left[
a,b\right] \right\} ,\text{ \ }0<\gamma <1,
\end{equation*}%
\begin{equation*}
C_{0}\left[ a,b\right] =C\left[ a,b\right] ,
\end{equation*}%
and%
\begin{equation}
C_{1-\alpha }^{\alpha }\left[ a,b\right] =\left\{ f\in C_{1-\alpha
}\left[
a,b\right] :D_{a}^{\alpha }f\in C_{1-\alpha }\left[ a,b\right] \right\} ,%
\text{ \ }0<\alpha <1.  \label{space}
\end{equation}
\end{definition}

\begin{lemma}
\textbf{\label{lem:7}}\textit{Let }$0\leq \gamma <1$\textit{\ and
}$f\in
C_{\gamma }\left[ a,b\right] $. \textit{Then}%
\begin{equation*}
I_{a}^{\alpha }f\left( a\right) =\underset{t\rightarrow a}{\lim }%
I_{a}^{\alpha }f\left( t\right) =0,\quad 0\leq \gamma <\alpha .
\end{equation*}
\end{lemma}

\begin{proof}
Since $f\in C_{\gamma }\left[ a,b\right] $ then $\left( t-a\right)
^{\gamma }f\left( t\right) $ is continuous on $[a,b]$ and on
$[a,b]$ we have
\begin{equation*}
\left\vert \left( t-a\right) ^{\gamma }f\left( t\right)
\right\vert <M,
\end{equation*}%
for some positive constant $M$. Therefore
\begin{equation*}
\left\vert I_{a}^{\alpha }f\left( t\right) \right\vert <M\left[
I_{a}^{\alpha }\left( s-a\right) ^{-\gamma }\right] \left( t\right) =M\frac{%
\Gamma \left( 1-\gamma \right) }{\Gamma \left( \alpha +1-\gamma \right) }%
\left( t-a\right) ^{\alpha -\gamma }.
\end{equation*}%
As $\alpha >\gamma $ we see that
\begin{equation*}
I_{a}^{\alpha }f\left( a\right) =\underset{t\rightarrow a}{\lim }%
I_{a}^{\alpha }f\left( t\right) =0,\quad 0\leq \gamma <\alpha
\end{equation*}%
which completes the proof of Lemma \textbf{\ref{lem:7}}.
\end{proof}

\begin{lemma}
\textbf{\label{lem:8}}\textit{Let }$\varphi \in C^{1}[0,\infty
)$\textit{\
be a test function, that is: }$\varphi \left( t\right) \geq 0,$\textit{\ }$%
\varphi \left( t\right) $\textit{\ is non-increasing and such that }%
\begin{equation*}
\begin{array}{ll}
\varphi \left( t\right) := & \left\{
\begin{array}{l}
1,\quad t\in \left[ 0,T/2\right] \\
0,\quad t\in \lbrack T,\infty ),%
\end{array}%
\right.%
\end{array}%
\end{equation*}%
\textit{for }$T>0$\textit{. Then}%
\begin{equation}
I\left( T\right) =\int\nolimits_{T/2}^{T}\left( I_{T-}^{1-\alpha }\frac{%
\left\vert \varphi ^{\prime }\right\vert }{\varphi ^{p}}\right)
^{m}\left(
t\right) dt\leq K_{\alpha ,m}T^{1-\alpha m},\text{ \ }0<\alpha <1\text{,\ }T,%
\text{ }p,\text{ }m>0  \label{t1}
\end{equation}%
\textit{where}%
\begin{equation}
K_{\alpha ,m}=\frac{K_{1}^{m}}{2^{m\left( 1-\alpha \right)
+1}\Gamma ^{m}\left( 2-\alpha \right) \left[ m\left( 1-\alpha
\right) +1\right] }, \label{t2}
\end{equation}%
\textit{and }$K_{1}$\textit{\ is a bound for }$\frac{\left\vert
\varphi
^{\prime }\left( r\right) \right\vert }{\varphi \left( r\right) ^{p}}.$%
\textit{\newline }
\end{lemma}

\begin{proof}
Using (\ref{ri2}), we see that%
\begin{equation}
I\left( T\right) =\int\nolimits_{T/2}^{T}\left( \frac{1}{\Gamma
\left(
1-\alpha \right) }\int\nolimits_{t}^{T}\left( s-t\right) ^{-\alpha }\frac{%
\left\vert \varphi ^{\prime }\left( s\right) \right\vert }{\varphi
\left( s\right) ^{p}}ds\right) ^{m}dt.  \label{t3}
\end{equation}%
The change of variable $\sigma T=t$ in (\ref{t3}) yields%
\begin{equation}
I\left( T\right) =\int\nolimits_{1/2}^{1}\left( \frac{1}{\Gamma
\left( 1-\alpha \right) }\int\nolimits_{\sigma T}^{T}\left(
s-\sigma T\right)
^{-\alpha }\frac{\left\vert \varphi ^{\prime }\left( s\right) \right\vert }{%
\varphi \left( s\right) ^{p}}ds\right) ^{m}Td\sigma .  \label{t4}
\end{equation}%
Another change of variable $s=rT$ in (\ref{t4}) gives%
\begin{equation*}
I\left( T\right) =\int\nolimits_{1/2}^{1}\left( \frac{1}{\Gamma
\left( 1-\alpha \right) }\int\nolimits_{\sigma }^{1}\left(
rT-\sigma T\right)
^{-\alpha }\frac{\left\vert \varphi ^{\prime }\left( r\right) \right\vert }{%
\varphi \left( r\right) ^{p}}dr\right) ^{m}Td\sigma
\end{equation*}%
\begin{equation}
=\frac{T^{1-\alpha m}}{\Gamma ^{m}\left( 1-\alpha \right) }%
\int\nolimits_{1/2}^{1}\left( \int\nolimits_{\sigma }^{1}\left(
r-\sigma \right) ^{-\alpha }\frac{\left\vert \varphi ^{\prime
}\left( r\right) \right\vert }{\varphi \left( r\right)
^{p}}dr\right) ^{m}d\sigma . \label{t5}
\end{equation}%
Since $\varphi \in C^{1}\left( [0,\infty )\right) ,$ we may assume
without
loss of generality that%
\begin{equation*}
\frac{\left\vert \varphi ^{\prime }\left( r\right) \right\vert
}{\varphi \left( r\right) ^{p}}\leq K_{1},
\end{equation*}%
for some positive constant $K_{1}$, for otherwise we consider
$\varphi ^{\lambda }\left( r\right) $ with some sufficiently large
$\lambda $.
Therefore from (\ref{t5}) we get%
\begin{equation*}
I\left( T\right) \leq \frac{K_{1}^{m}T^{1-\alpha m}}{\Gamma
^{m}\left( 1-\alpha \right) }\int\nolimits_{1/2}^{1}\left(
\int\nolimits_{\sigma
}^{1}\left( r-\sigma \right) ^{-\alpha }dr\right) ^{m}d\sigma =\frac{%
K_{1}^{m}T^{1-\alpha m}}{\Gamma ^{m}\left( 2-\alpha \right) }%
\int\nolimits_{1/2}^{1}\left( 1-\sigma \right) ^{m\left( 1-\alpha
\right) }d\sigma
\end{equation*}%
\begin{equation*}
=\frac{K_{1}^{m}}{2^{m\left( 1-\alpha \right) +1}\Gamma ^{m}\left(
2-\alpha \right) \left[ m\left( 1-\alpha \right) +1\right]
}T^{1-\alpha m}.
\end{equation*}%
Therefore%
\begin{equation*}
I\left( T\right) \leq K_{\alpha ,m}T^{1-\alpha m}.
\end{equation*}
\end{proof}

\begin{remark}
\label{rem:1}Lemma \textbf{\ref{lem:8}} is true also for the case
$\alpha =1$ . We prove this fact in the following lemma.
\end{remark}

\begin{lemma}
\label{lem:9}\textit{Let }$\varphi $ \textit{be as in Lemma \textbf{\ref%
{lem:8}}. Then}%
\begin{equation}
I\left( T\right) =\int\nolimits_{T/2}^{T}\left( \frac{\left\vert
\varphi ^{\prime }\left( t\right) \right\vert }{\varphi ^{p}\left(
t\right) }\right) ^{m}dt\leq \frac{1}{2}K_{1}^{m}T^{1-m},\text{ \
\ }T,\text{ }p,\text{ }m>0, \label{t6}
\end{equation}%
\textit{with}%
\begin{equation*}
\frac{\left\vert \varphi ^{\prime }\left( r\right) \right\vert
}{\varphi \left( r\right) ^{p}}\leq K_{1}.
\end{equation*}
\end{lemma}

\begin{proof}
The change of variable $sT=t$ in the expression of $I\left(
T\right) $\
leads to%
\begin{equation*}
I\left( T\right) =\int\nolimits_{1/2}^{1}\left( \frac{\left\vert
\varphi
^{\prime }\left( s\right) \right\vert }{T\varphi ^{p}\left( s\right) }%
\right) ^{m}Tds=T^{1-m}\int\nolimits_{1/2}^{1}\left(
\frac{\left\vert
\varphi ^{\prime }\left( s\right) \right\vert }{\varphi ^{p}\left( s\right) }%
\right) ^{m}ds\leq \frac{1}{2}K_{1}^{m}T^{1-m}.
\end{equation*}
\end{proof}

\section{Nonexistence result\label{non-ex}}

In this section, we consider the problem%
\begin{equation}
\left\{
\begin{array}{l}
D_{0}^{\alpha }y\left( t\right) +D_{0}^{\beta }y\left( t\right)
\geq
t^{\gamma }\left\vert y\left( t\right) \right\vert ^{m},\text{ }t>0,\text{\ }%
m>1,\text{\ \ \ }0<\beta <\alpha <1, \\
I_{0}^{1-\alpha }y\left( t\right) |_{t=0}=b,%
\end{array}%
\right.  \label{nr1}
\end{equation}%
where $D_{0}^{\sigma }$ is defined in (\ref{rd1}). Nonexistence of
non-trivial solutions is investigated in the space $C_{1-\alpha
}^{\alpha }$ defined in (\ref{space}).

\begin{theorem}
\label{thm:10}\textit{Assume that }$\gamma >-\beta $\textit{\ and
}$1<m\leq \frac{\gamma +1}{1-\beta }.$\textit{\ Then, Problem
(\ref{nr1}) does not admit global nontrivial solutions in
}$C_{1-\alpha }^{\alpha },$\textit{\ when }$b\geq 0.$
\end{theorem}

\begin{proof}
Assume, on the contrary, that a nontrivial solution $y$ exists for all time $%
t>0.$ Let $\varphi $ be as in Lemma \textbf{\ref{lem:8}}.
Multiplying the inequality in (\ref{nr1}) by $\varphi \left(
t\right) $ and integrating over $\left( 0,T\right) $ we get
\begin{equation}
I_{1}:=\int\nolimits_{0}^{T}t^{\gamma }\left\vert y\left( t\right)
\right\vert ^{m}\varphi \left( t\right) dt\leq
\int_{0}^{T}D_{0}^{\alpha }y\left( t\right) \varphi \left(
t\right) dt+\int_{0}^{T}D_{0}^{\beta }y\left( t\right) \varphi
\left( t\right) dt.  \label{nr2}
\end{equation}%
Let%
\begin{equation*}
I_{2}:=\int_{0}^{T}\varphi \left( t\right) D_{0}^{\alpha }y\left(
t\right) dt,
\end{equation*}%
and%
\begin{equation*}
I_{3}:=\int_{0}^{T}\varphi \left( t\right) D_{0}^{\beta }y\left(
t\right) dt.
\end{equation*}%
From the definition of $D_{0}^{\alpha }y$ in (\ref{rd1}) we can
write
\begin{equation*}
I_{2}=\int_{0}^{T}\varphi \left( t\right)
\frac{d}{dt}I_{0}^{1-\alpha }y\left( t\right) dt.
\end{equation*}%
An integration by parts yields
\begin{equation*}
I_{2}=\left[ \varphi \left( t\right) I_{0}^{1-\alpha }y\left( t\right) %
\right] _{t=0}^{T}-\int_{0}^{T}\varphi ^{\prime }\left( t\right)
I_{0}^{1-\alpha }y\left( t\right) dt.
\end{equation*}%
Since $\varphi \left( T\right) =0,$ $\varphi \left( 0\right) =1$ and $%
I_{0}^{1-\alpha }y\left( 0\right) =b,$ then
\begin{equation*}
I_{2}=-b-\int\nolimits_{0}^{T}\varphi ^{\prime }\left( t\right)
I_{0}^{1-\alpha }y\left( t\right) dt.
\end{equation*}%
As $b\geq 0,$ we have%
\begin{equation*}
I_{2}\leq -\int\nolimits_{0}^{T}\varphi ^{\prime }\left( t\right)
I_{0}^{1-\alpha }y\left( t\right) dt\leq
\int\nolimits_{0}^{T}\left\vert \varphi ^{\prime }\left( t\right)
\right\vert \left( I_{0}^{1-\alpha }\left\vert y\right\vert
\right) \left( t\right) dt
\end{equation*}%
\begin{equation}
\leq \frac{1}{\Gamma \left( 1-\alpha \right) }\int\nolimits_{0}^{T}\left%
\vert \varphi ^{\prime }\left( t\right) \right\vert \int\nolimits_{0}^{t}%
\frac{\left\vert y\left( s\right) \right\vert }{\left( t-s\right) ^{\alpha }}%
dsdt.\text{ \ \ \ \ \ \ \ \ \ }  \label{bound}
\end{equation}%
Because $\varphi \left( t\right) $ is nonincreasing $\varphi
\left( s\right)
\geq \varphi \left( t\right) $ for all $t\geq s,$ and therefore%
\begin{equation*}
\frac{1}{\varphi \left( s\right) ^{1/m}}\leq \frac{1}{\varphi
\left( t\right) ^{1/m}},\;m>1.
\end{equation*}%
Also we have
\begin{equation*}
\varphi ^{\prime }\left( t\right) =0,\;t\in \left[ 0,T/2\right] .
\end{equation*}%
Thus%
\begin{equation*}
I_{2}\leq \frac{1}{\Gamma \left( 1-\alpha \right) }\int\nolimits_{0}^{T}%
\left\vert \varphi ^{\prime }\left( t\right) \right\vert
\int\nolimits_{0}^{t}\frac{\left\vert y\left( s\right) \right\vert
}{\left( t-s\right) ^{\alpha }}\frac{\varphi \left( s\right)
^{1/m}}{\varphi \left( s\right) ^{1/m}}dsdt\text{ \ \ \ \ \ }
\end{equation*}%
\begin{equation*}
\leq \frac{1}{\Gamma \left( 1-\alpha \right) }\int\nolimits_{0}^{T}\frac{%
\left\vert \varphi ^{\prime }\left( t\right) \right\vert }{\varphi
\left( t\right) ^{1/m}}\int\nolimits_{0}^{t}\frac{\left\vert
y\left( s\right)
\right\vert }{\left( t-s\right) ^{\alpha }}\varphi \left( s\right) ^{1/m}dsdt%
\text{ \ }
\end{equation*}%
\begin{equation*}
\leq \frac{1}{\Gamma \left( 1-\alpha \right) }\int\nolimits_{T/2}^{T}\frac{%
\left\vert \varphi ^{\prime }\left( t\right) \right\vert }{\varphi
\left( t\right) ^{1/m}}\int\nolimits_{0}^{t}\frac{\left\vert
y\left( s\right) \right\vert }{\left( t-s\right) ^{\alpha
}}\varphi \left( s\right) ^{1/m}dsdt
\end{equation*}%
\begin{equation*}
\leq \int\nolimits_{T/2}^{T}\frac{\left\vert \varphi ^{\prime
}\left( t\right) \right\vert }{\varphi \left( t\right)
^{1/m}}\left( I_{0}^{1-\alpha }\varphi ^{1/m}\left\vert
y\right\vert \right) \left( t\right) dt.\text{ \ \ \ \ \ \ \ \ \ \
\ \ \ \ \ \ \ \ \ \ \ \ \ \ }
\end{equation*}%
A fractional integration by parts (\ref{fibp}), in the last
expression yields
\begin{equation*}
I_{2}\leq \int\nolimits_{T/2}^{T}\left( I_{T-}^{1-\alpha
}\frac{\left\vert \varphi ^{\prime }\right\vert }{\varphi
^{1/m}}\right) \left( t\right) \varphi \left( t\right)
^{1/m}\left\vert y\left( t\right) \right\vert dt.
\end{equation*}%
Next, we multiply by $t^{\gamma /m}t^{-\gamma /m}$ inside the
integral in the right hand side
\begin{equation*}
I_{2}\leq \int\nolimits_{T/2}^{T}\left( I_{T-}^{1-\alpha
}\frac{\left\vert \varphi ^{\prime }\right\vert }{\varphi
^{1/m}}\right) \left( t\right) \varphi \left( t\right)
^{1/m}\frac{t^{\gamma /m}}{t^{\gamma /m}}\left\vert y\left(
t\right) \right\vert dt.
\end{equation*}%
For $\gamma <0$ we have $t^{-\gamma /m}<T^{-\gamma /m}$ (because
$t<T$) and for $\gamma >0$ we get $t^{-\gamma /m}<2^{\gamma
/m}T^{-\gamma /m}$ (because $T/2<t$): that is
\begin{equation*}
t^{-\gamma /m}<\max \left\{ 1,2^{\gamma /m}\right\} T^{-\gamma
/m}.
\end{equation*}%
Then
\begin{equation}
I_{2}\leq \max \left\{ 1,2^{\gamma /m}\right\} T^{-\gamma
/m}\int\nolimits_{T/2}^{T}\left( I_{T-}^{1-\alpha
}\frac{\left\vert \varphi ^{\prime }\right\vert }{\varphi
^{1/m}}\right) \left( t\right) t^{\gamma /m}\varphi \left(
t\right) ^{1/m}\left\vert y\left( t\right) \right\vert dt.
\label{nr3}
\end{equation}%
By H\"{o}lder's inequality, it is clear that%
\begin{equation*}
I_{2}\leq \max \left\{ 1,2^{\gamma /m}\right\} T^{-\gamma
/m}\left( \int\nolimits_{T/2}^{T}t^{\gamma }\varphi \left(
t\right) \left\vert y\left( t\right) \right\vert ^{m}dt\right)
^{\frac{1}{m}}\left( \int\nolimits_{T/2}^{T}\left(
I_{T-}^{1-\alpha }\frac{\left\vert \varphi ^{\prime }\right\vert
}{\varphi ^{1/m}}\right) ^{m^{\prime }}\left( t\right) dt\right)
^{\frac{1}{m^{\prime }}}.
\end{equation*}%
Lemma \textbf{\ref{lem:8}} implies that%
\begin{equation}
I_{2}\leq \max \left\{ 1,2^{\gamma /m}\right\} T^{-\gamma
/m}\left( \int\nolimits_{T/2}^{T}t^{\gamma }\varphi \left(
t\right) \left\vert y\left( t\right) \right\vert ^{m}dt\right)
^{\frac{1}{m}}\left( K_{\alpha ,m^{\prime }}T^{1-\alpha m^{\prime
}}\right) ^{\frac{1}{m^{\prime }}},  \label{nr4}
\end{equation}%
where $K_{\alpha ,m^{\prime }}$ is the constant appearing in Lemma \textbf{%
\ref{lem:8}} corresponding to the present exponents. Therefore from (\ref%
{nr4}) we have the estimate%
\begin{equation}
I_{2}\leq \max \left\{ 1,2^{\gamma /m}\right\} K_{\alpha ,m^{\prime }}^{%
\frac{1}{m^{\prime }}}T^{1/m^{\prime }-\alpha -\gamma /m}I_{1}^{\frac{1}{m}}.%
\text{\ \ \ \ \ }  \label{nr5}
\end{equation}%
Now, we turn to $I_{3}$. First, since $y\in C_{1-\alpha }\left[
0,T\right] $
and $1-\alpha <1-\beta $, then by Lemma \textbf{\ref{lem:7}} we have%
\begin{equation*}
I_{0}^{1-\beta }y\left( 0\right) =\lim_{t\rightarrow
0}I_{0}^{1-\beta }y\left( t\right) =0.
\end{equation*}%
An integration by parts in%
\begin{equation*}
I_{3}=\int_{0}^{T}\varphi \left( t\right) D_{0}^{\beta }y\left(
t\right) dt=\int_{0}^{T}\varphi \left( t\right)
\frac{d}{dt}I_{0}^{1-\beta }y\left( t\right) dt
\end{equation*}%
gives
\begin{equation*}
I_{3}=\left[ \varphi \left( t\right) I_{0}^{1-\beta }y\left(
t\right) \right] _{t=0}^{T}-\int_{0}^{T}\varphi ^{\prime }\left(
t\right) I_{0}^{1-\beta }y\left( t\right) dt.
\end{equation*}%
Since $\varphi \left( T\right) =0$ and $I_{0}^{1-\beta }y\left(
0\right) =0,$
it follows that%
\begin{equation*}
I_{3}=-\int\nolimits_{0}^{T}\varphi ^{\prime }\left( t\right)
I_{0}^{1-\beta }y\left( t\right) dt\leq
\int\nolimits_{0}^{T}\left\vert \varphi ^{\prime }\left( t\right)
\right\vert \left( I_{0}^{1-\beta }\left\vert y\right\vert \right)
\left( t\right) dt
\end{equation*}%
\begin{equation*}
\leq \frac{1}{\Gamma \left( 1-\beta \right)
}\int\nolimits_{0}^{T}\left\vert
\varphi ^{\prime }\left( t\right) \right\vert \int\nolimits_{0}^{t}\frac{%
\left\vert y\left( s\right) \right\vert }{\left( t-s\right) ^{\beta }}dsdt.%
\text{ \ \ \ \ \ \ \ \ \ \ }
\end{equation*}%
Replacing $\alpha $ by $\beta $ in the argument above allows us to write%
\begin{equation}
I_{3}\leq \max \left\{ 1,2^{\gamma /m}\right\} T^{-\gamma
/m}\int\nolimits_{T/2}^{T}\left( I_{T-}^{1-\beta }\frac{\left\vert
\varphi ^{\prime }\right\vert }{\varphi ^{1/m}}\right) \left(
t\right) t^{\gamma /m}\varphi \left( t\right) ^{1/m}\left\vert
y\left( t\right) \right\vert dt, \label{nr6}
\end{equation}%
or simply
\begin{equation}
I_{3}\leq K_{\beta ,m^{\prime }}^{\frac{1}{m^{\prime }}}\max
\left\{
1,2^{\gamma /m}\right\} T^{1/m^{\prime }-\beta -\gamma /m}I_{1}^{\frac{1}{m}%
}.  \label{nr7}
\end{equation}%
From (\ref{nr2}), (\ref{nr5}) and (\ref{nr7}), we have%
\begin{equation*}
I_{1}\leq \max \left\{ 1,2^{\gamma /m}\right\} K_{\alpha ,m^{\prime }}^{%
\frac{1}{m^{\prime }}}T^{1/m^{\prime }-\alpha -\gamma /m}I_{1}^{\frac{1}{m}%
}+K_{\beta ,m^{\prime }}^{\frac{1}{m^{\prime }}}\max \left\{
1,2^{\gamma /m}\right\} T^{1/m^{\prime }-\beta -\gamma
/m}I_{1}^{\frac{1}{m}}
\end{equation*}%
\begin{equation*}
\leq \max \left\{ K_{\alpha ,m^{\prime }}^{\frac{1}{m^{\prime
}}},K_{\beta ,m^{\prime }}^{\frac{1}{m^{\prime }}}\right\} \max
\left\{ 1,2^{\gamma /m}\right\} \left( T^{1/m^{\prime }-\alpha
-\gamma /m}+T^{1/m^{\prime }-\beta -\gamma /m}\right)
I_{1}^{\frac{1}{m}}.
\end{equation*}%
Therefore
\begin{equation}
I_{1}^{\frac{1}{m^{\prime }}}\leq K_{2}\left( T^{1/m^{\prime
}-\alpha -\gamma /m}+T^{1/m^{\prime }-\beta -\gamma /m}\right) ,
\label{nr8.1}
\end{equation}%
with
\begin{equation*}
K_{2}:=\max \left\{ K_{\alpha ,m^{\prime }}^{\frac{1}{m^{\prime
}}},K_{\beta ,m^{\prime }}^{\frac{1}{m^{\prime }}}\right\} \max
\left\{ 1,2^{\gamma /m}\right\} .
\end{equation*}%
Raising both sides of (\ref{nr8.1}) to the power $m^{\prime }$ we obtain%
\begin{equation}
I_{1}\leq K_{3}\left( T^{1-\alpha m^{\prime }-\gamma m^{\prime
}/m}+T^{1-\beta m^{\prime }-\gamma m^{\prime }/m}\right) ,
\label{nr8}
\end{equation}%
with%
\begin{equation*}
K_{3}=2^{1-m^{\prime }}K_{2}^{m^{\prime }}.
\end{equation*}%
If $m<\frac{\gamma +1}{1-\beta }$ we see that $1-\beta m^{^{\prime
}}-\gamma m^{^{\prime }}/m<0,$ $1-\alpha m^{^{\prime }}-\gamma
m^{^{\prime }}/m<0,$ and consequently $T^{1-\beta m^{^{\prime
}}-\gamma m^{^{\prime }}/m}\rightarrow 0$ and $T^{1-\alpha
m^{^{\prime }}-\gamma m^{^{\prime }}/m}\rightarrow 0$ as
$T\rightarrow \infty $. Then, from (\ref{nr8}), we obtain
\begin{equation*}
\underset{T\rightarrow \infty }{\lim I_{1}=}\underset{T\rightarrow \infty }{%
\lim }\int\nolimits_{0}^{T}t^{\gamma }\left\vert y\left( t\right)
\right\vert ^{m}\varphi \left( t\right) dt=0.
\end{equation*}%
We reach a contradiction since the solution is not supposed to be trivial.%
\newline
In the case $m=\frac{\gamma +1}{1-\beta }$ we have $1-\beta
m^{^{\prime }}-\gamma m^{^{\prime }}/m$ $=0,$ $1-\alpha
m^{^{\prime }}-\gamma m^{^{\prime }}/m\leq 0,$\ and the relation
(\ref{nr8}) ensures that
\begin{equation}
\underset{T\rightarrow \infty }{\lim
}\int\nolimits_{0}^{T}t^{\gamma }\left\vert y\left( t\right)
\right\vert ^{m}\varphi \left( t\right) dt\leq K_{4}.  \label{nr9}
\end{equation}%
Further, in view of (\ref{nr2}), (\ref{nr3}) and (\ref{nr6}), we see that%
\begin{equation*}
I_{1}\leq \max \left\{ 1,2^{\gamma /m}\right\} T^{-\gamma
/m}\int\nolimits_{T/2}^{T}t^{\gamma /m}\varphi \left( t\right)
^{1/m}\left\vert y\left( t\right) \right\vert \left[ \left(
I_{T-}^{1-\alpha }\frac{\left\vert \varphi ^{\prime }\right\vert
}{\varphi ^{1/m}}\right) \left( t\right) +\left( I_{T-}^{1-\beta
}\frac{\left\vert \varphi ^{\prime }\right\vert }{\varphi
^{1/m}}\right) \left( t\right) \right] dt.
\end{equation*}%
Thanks to H\"{o}lder's inequality, it is clear that
\begin{equation*}
I_{1}\leq \max \left\{ 1,2^{\gamma /m}\right\} T^{-\gamma
/m}\left[ \int\nolimits_{T/2}^{T}t^{\gamma }\varphi \left(
t\right) \left\vert y\left( t\right) \right\vert ^{m}dt\right]
^{\frac{1}{m}}\text{ \ \ \ \ \ \ \ \ \ \ \ \ \ \ \ \ \ \ \ \ \ \ \
\ }
\end{equation*}%
\begin{equation*}
\times \left\{ \int\nolimits_{T/2}^{T}\left[ \left( I_{T-}^{1-\alpha }\frac{%
\left\vert \varphi ^{\prime }\right\vert }{\varphi ^{1/m}}\right)
\left( t\right) +\left( I_{T-}^{1-\beta }\frac{\left\vert \varphi
^{\prime }\right\vert }{\varphi ^{1/m}}\right) \left( t\right)
\right] ^{m^{^{\prime }}}dt\right\} ^{\frac{1}{m^{^{\prime }}}}
\end{equation*}%
\begin{equation*}
\leq \max \left\{ 1,2^{\gamma /m}\right\} 2^{1/m}T^{-\gamma
/m}\left[ \int\nolimits_{T/2}^{T}t^{\gamma }\varphi \left(
t\right) \left\vert y\left( t\right) \right\vert ^{m}dt\right]
^{\frac{1}{m}}
\end{equation*}%
\begin{equation*}
\times \left\{ \int\nolimits_{T/2}^{T}\left[ \left( I_{T-}^{1-\alpha }\frac{%
\left\vert \varphi ^{\prime }\right\vert }{\varphi ^{1/m}}\right)
^{m^{\prime }}\left( t\right) +\left( I_{T-}^{1-\beta
}\frac{\left\vert \varphi ^{\prime }\right\vert }{\varphi
^{1/m}}\right) ^{m^{\prime }}\left( t\right) \right] dt\right\}
^{\frac{1}{m^{^{\prime }}}}.
\end{equation*}%
Therefore, by Lemma \textbf{\ref{lem:8}}, we obtain%
\begin{equation*}
I_{1}\leq K_{5}T^{-\gamma /m}\left[
\int\nolimits_{T/2}^{T}t^{\gamma
}\varphi \left( t\right) \left\vert y\left( t\right) \right\vert ^{m}dt%
\right] ^{\frac{1}{m}}\left[ K_{\alpha ,m^{\prime }}T^{1-\alpha
m^{\prime
}}+K_{\beta ,m^{\prime }}T^{1-\beta m^{\prime }}\right] ^{\frac{1}{%
m^{^{\prime }}}}
\end{equation*}%
\begin{equation*}
=K_{5}\left[ \int\nolimits_{T/2}^{T}t^{\gamma }\varphi \left(
t\right) \left\vert y\left( t\right) \right\vert ^{m}dt\right]
^{\frac{1}{m}}\left[ K_{\alpha ,m^{\prime }}T^{1-\alpha m^{\prime
}-\gamma m^{^{\prime
}}/m}+K_{\beta ,m^{\prime }}T^{1-\beta m^{\prime }-\gamma m^{^{\prime }}/m}%
\right] ^{\frac{1}{m^{^{\prime }}}},
\end{equation*}%
with%
\begin{equation*}
K_{5}=\max \left\{ 1,2^{\gamma /m}\right\} 2^{1/m}.
\end{equation*}%
Since $m=\frac{\gamma +1}{1-\beta },$ then $1-\beta m^{^{\prime
}}-\gamma m^{^{\prime }}/m$ $=0$ and $1-\alpha m^{^{\prime
}}-\gamma m^{^{\prime
}}/m\leq 0.$ Therefore%
\begin{equation*}
I_{1}\leq K_{6}\left[ \int\nolimits_{T/2}^{T}t^{\gamma }\varphi
\left( t\right) \left\vert y\left( t\right) \right\vert
^{m}dt\right] ^{\frac{1}{m}}
\end{equation*}%
for some positive constant $K_{6}$, with
\begin{equation*}
\underset{T\rightarrow \infty }{\lim
}\int\nolimits_{T/2}^{T}t^{\gamma }\varphi \left( t\right)
\left\vert y\left( t\right) \right\vert ^{m}dt=0
\end{equation*}%
due to the convergence of the integral in (\ref{nr9}). This is
again a contradiction. The proof is complete.
\end{proof}

\ \ \newline
Next, we take $\alpha =1$ and $0<\beta <1,$ that is%
\begin{equation}
\left\{
\begin{array}{l}
y^{\prime }\left( t\right) +D_{0}^{\beta }y\left( t\right) \geq
t^{\gamma
}\left\vert y\left( t\right) \right\vert ^{m},\text{ }t>0,\text{\ }m>1,\text{%
\ \ \ }0<\beta <1, \\
y\left( t\right) |_{t=0}=b\in
%TCIMACRO{\U{211d} }%
%BeginExpansion
\mathbb{R}
%EndExpansion
.%
\end{array}%
\right.  \label{nr10}
\end{equation}

\begin{theorem}
\textbf{\label{thm:11}}\textit{Assume that\ }$\gamma >-\beta $\textit{\ and }%
$1<m\leq \frac{\gamma +1}{1-\beta }.$\textit{\ Then, Problem
(\ref{nr10}) does not admit global nontrivial solutions\ when
}$b\geq 0.$
\end{theorem}

\begin{proof}
Assume, on the contrary, that a nontrivial solution $y$ exists for all time $%
t>0.$ Let $\varphi $ be as in Lemma \textbf{\ref{lem:8}}.
Multiplying the inequality in (\ref{nr10}) by $\varphi \left(
t\right) $ and integrating we get
\begin{equation}
J_{1}=\int\nolimits_{0}^{T}t^{\gamma }\left\vert y\left( t\right)
\right\vert ^{m}\varphi \left( t\right) dt\leq
\int_{0}^{T}y^{\prime }\left( t\right) \varphi \left( t\right)
dt+\int_{0}^{T}D_{0}^{\beta }y\left( t\right) \varphi \left(
t\right) dt.  \label{nr11}
\end{equation}%
Let%
\begin{equation}
J_{2}=\int_{0}^{T}\varphi \left( t\right) y^{\prime }\left(
t\right) dt, \label{nr13}
\end{equation}%
and%
\begin{equation}
J_{3}=\int_{0}^{T}\varphi \left( t\right) D_{0}^{\beta }y\left(
t\right) dt. \label{nr14}
\end{equation}%
Following procedure as in the proof of Theorem
\textbf{\ref{thm:10}}, we
obtain the following estimates for $J_{2}$ and $J_{3}$%
\begin{equation}
J_{2}\leq \max \left\{ 1,2^{\gamma /m}\right\} T^{-\gamma
/m}\int\nolimits_{T/2}^{T}\frac{\left\vert \varphi ^{\prime
}\left( t\right) \right\vert }{\varphi \left( t\right)
^{1/m}}\left\vert y\left( t\right) \right\vert \varphi \left(
t\right) ^{1/m}t^{\gamma /m}dt,  \label{nr15}
\end{equation}%
(or By using H\"{o}lder's inequality and Lemma \textbf{\ref{lem:9}})%
\begin{equation}
J_{2}\leq \max \left\{ 1,2^{\gamma /m}\right\} K_{1}T^{1/m^{\prime
}-1-\gamma /m}J_{1}^{\frac{1}{m}},\text{\ \ \ \ }  \label{nr17}
\end{equation}%
and%
\begin{equation}
J_{3}\leq \max \left\{ 1,2^{\gamma /m}\right\} T^{-\gamma
/m}\int\nolimits_{T/2}^{T}\left( I_{T-}^{1-\beta }\frac{\left\vert
\varphi ^{\prime }\right\vert }{\varphi ^{1/m}}\right) \left(
t\right) t^{\gamma /m}\varphi \left( t\right) ^{1/m}\left\vert
y\left( t\right) \right\vert dt, \label{nr18}
\end{equation}%
(or By using H\"{o}lder's inequality and Lemma
\textbf{\ref{lem:8}})
\begin{equation}
J_{3}\leq \max \left\{ 1,2^{\gamma /m}\right\} K_{\beta ,m^{\prime }}^{\frac{%
1}{m^{\prime }}}T^{1/m^{\prime }-\beta -\gamma
/m}J_{1}^{\frac{1}{m}}. \label{nr19}
\end{equation}%
From (\ref{nr11}), (\ref{nr17}) and (\ref{nr19}), we have
\begin{equation}
J_{1}^{\frac{1}{m^{\prime }}}\leq K_{2}\left( T^{1/m^{\prime
}-1-\gamma /m}+T^{1/m^{\prime }-\beta -\gamma /m}\right) ,
\label{nr20}
\end{equation}%
with
\begin{equation*}
K_{2}:=\max \left\{ 1,2^{\gamma /m}\right\} \max \left\{ K_{\beta
,m^{\prime }}^{\frac{1}{m^{\prime }}},K_{1}\right\} \text{ }.
\end{equation*}%
Raising both sides of (\ref{nr20}) to the power $m^{\prime }$ we obtain%
\begin{equation}
J_{1}\leq K_{3}\left( T^{1-m^{\prime }-\gamma m^{\prime
}/m}+T^{1-\beta m^{\prime }-\gamma m^{\prime }/m}\right) ,
\label{nr21}
\end{equation}%
with%
\begin{equation*}
K_{3}=2^{1-m^{\prime }}K_{2}^{m^{\prime }}.
\end{equation*}%
If $m<\frac{\gamma +1}{1-\beta }$ we see that $1-m^{\prime
}-\gamma m^{^{\prime }}/m<0,$ $1-\beta m^{^{\prime }}-\gamma
m^{^{\prime }}/m<0$. Then from (\ref{nr21}) we obtain
\begin{equation*}
\underset{T\rightarrow \infty }{\lim J_{1}=}\underset{T\rightarrow \infty }{%
\lim }\int\nolimits_{0}^{T}t^{\gamma }\left\vert y\left( t\right)
\right\vert ^{m}\varphi \left( t\right) dt=0.
\end{equation*}%
We reach a contradiction since the solution is not supposed to be trivial.%
\newline
In the case $m=\frac{\gamma +1}{1-\beta }$ we have $1-m^{\prime
}-\gamma m^{^{\prime }}/m$ $\leq 0,$ $1-\beta m^{^{\prime
}}-\gamma m^{^{\prime }}/m=0,$\ and the relation (\ref{nr21})
ensures that
\begin{equation}
\underset{T\rightarrow \infty }{\lim
}\int\nolimits_{0}^{T}t^{\gamma }\left\vert y\left( t\right)
\right\vert ^{m}\varphi \left( t\right) dt\leq K_{4}.
\label{nr22}
\end{equation}%
Also from (\ref{nr11}), (\ref{nr15}) and (\ref{nr18}), we have%
\begin{equation*}
J_{1}\leq \max \left\{ 1,2^{\gamma /m}\right\} T^{-\gamma
/m}\int\nolimits_{T/2}^{T}t^{\gamma /m}\varphi \left( t\right)
^{1/m}\left\vert y\left( t\right) \right\vert \left[
\frac{\left\vert \varphi ^{\prime }\left( t\right) \right\vert
}{\varphi \left( t\right) ^{1/m}}+\left( I_{T-}^{1-\beta
}\frac{\left\vert \varphi ^{\prime }\right\vert }{\varphi
^{1/m}}\right) \left( t\right) \right] dt.
\end{equation*}%
By using H\"{o}lder's inequality, it is clear that
\begin{equation*}
J_{1}\leq \max \left\{ 1,2^{\gamma /m}\right\} T^{-\gamma
/m}\left[ \int\nolimits_{T/2}^{T}t^{\gamma }\varphi \left(
t\right) \left\vert y\left( t\right) \right\vert ^{m}dt\right]
^{\frac{1}{m}}\text{ \ \ \ }
\end{equation*}%
\begin{equation*}
\times \left[ \int\nolimits_{T/2}^{T}\left[ \frac{\left\vert
\varphi
^{\prime }\left( t\right) \right\vert }{\varphi \left( t\right) ^{1/m}}%
+\left( I_{T-}^{1-\beta }\frac{\left\vert \varphi ^{\prime }\right\vert }{%
\varphi ^{1/m}}\right) \left( t\right) \right] ^{m^{^{\prime }}}dt\right] ^{%
\frac{1}{m^{^{\prime }}}}.
\end{equation*}%
Therefore, by Lemma \textbf{\ref{lem:8}} and Lemma \textbf{\ref{lem:9}} and $%
\varphi \in C^{1}[0,\infty ),$ we have%
\begin{equation*}
J_{1}\leq \max \left\{ 1,2^{\gamma /m}\right\} T^{-\gamma
/m}\left[ \int\nolimits_{T/2}^{T}t^{\gamma }\varphi \left(
t\right) \left\vert y\left( t\right) \right\vert ^{m}dt\right]
^{\frac{1}{m}}\text{ \ \ }\left[ K_{5}T^{1-m^{\prime
}}+K_{6}T^{1-\beta m^{\prime }}\right] ^{1/m^{\prime }},
\end{equation*}%
for some positive constants $K_{5}$ and $K_{6},$ and then%
\begin{equation*}
J_{1}\leq \max \left\{ 1,2^{\gamma /m}\right\} \left[ \int%
\nolimits_{T/2}^{T}t^{\gamma }\varphi \left( t\right) \left\vert
y\left( t\right) \right\vert ^{m}dt\right] ^{\frac{1}{m}}\text{ \
\ }\left[ K_{5}T^{1-m^{\prime }-\gamma m^{^{\prime
}}/m}+K_{6}T^{1-\beta m^{\prime }-\gamma m^{^{\prime }}/m}\right]
^{1/m^{\prime }}
\end{equation*}%
\begin{equation*}
\leq \max \left\{ 1,2^{\gamma /m}\right\} \left[ \int\nolimits_{T/2}^{T}t^{%
\gamma }\varphi \left( t\right) \left\vert y\left( t\right)
\right\vert ^{m}dt\right] ^{\frac{1}{m}}\left[ K_{5}T^{1-m^{\prime
}-\gamma m^{^{\prime }}/m}+K_{6}\right] ^{1/m^{\prime }},
\end{equation*}%
and%
\begin{equation*}
\underset{T\rightarrow \infty }{\lim
}\int\nolimits_{T/2}^{T}t^{\gamma }\varphi \left( t\right)
\left\vert y\left( t\right) \right\vert ^{m}dt=0
\end{equation*}%
due to the convergence of the integral in (\ref{nr22}). This is
again a contradiction and the proof of Theorem
\textbf{\ref{thm:11}} is complete.
\end{proof}

\ \ \newline
Finally, we take $\alpha =\beta =1,$ this mean we consider the Cauchy problem%
\begin{equation}
\left\{
\begin{array}{l}
y^{\prime }\left( t\right) \geq t^{\gamma }\left\vert y\left(
t\right)
\right\vert ^{m},\text{ }t>0,\text{\ \ \ }m>1, \\
y\left( t\right) |_{t=0}=b\in
%TCIMACRO{\U{211d} }%
%BeginExpansion
\mathbb{R}
%EndExpansion
.%
\end{array}%
\right.  \label{nr23}
\end{equation}%
\newline

\begin{theorem}
\textbf{\label{thm:12}}\textit{Assume that\ }$\gamma >-1$\textit{\ and }$%
m>1. $\textit{\ Then, Problem (\ref{nr23}) does not admit global
nontrivial solutions\ when }$b\geq 0.$
\end{theorem}

\begin{proof}
Similar to the proof of Theorem \textbf{\ref{thm:10}}.
\end{proof}

\begin{conclusion}
According to Theorems \textbf{\ref{thm:10}}, \textbf{\ref{thm:11}}
and having in mind the results in \cite{Laskri-Tatar(2010)} it
appears that the addition of the term $D_{0}^{\beta }y,$ $\beta
<\alpha ,$ does not prevent the nonexistence. However, it does
affect the exponent $m$. The range of $m$
is reduced to $1<m\leq \frac{\gamma +1}{1-\beta }$ instead of $1<m\leq \frac{%
\gamma +1}{1-\alpha }$. This shows that the range does not depend
on the highest derivative. It depends on the lowest derivative.
This is a well-established result for the Telegraph equation.
Indeed, for this problem, it has been proved that solutions
approach solutions of the corresponding parabolic part.\newline In
case $m$ is fixed from the beginning then we need $\gamma >m\left(
1-\beta \right) -1$ instead of $\gamma >m\left( 1-\alpha \right)
-1.$
Therefore, it is the derivative of lower order which determines the exponent.%
\newline
Note that%
\begin{equation*}
1<m\leq \frac{\gamma +1}{1-\beta }<\frac{\gamma +1}{1-\alpha },
\end{equation*}%
and
\end{conclusion}

\begin{equation*}
\gamma >-\beta >-\alpha .
\end{equation*}%
\textbf{Acknowledgement.} The authors would like to acknowledge
the support provided by King Fahd University of Petroleum and
Minerals (KFUPM) through project number IN151035.


\begin{thebibliography}{99}

\bibitem{A1} Agarwal R P, Benchohra M, Hamani S A. Survey on
existence results for boundary value problems of nonlinear
fractional differential equations and inclusions. Acta Appl.
Math., 2010, 109, 973-1033.

\bibitem{A2} Agarwal R P, Belmekki M, Benchohra M. A survey on
semilinear differential equations and inclusions involving
Riemann--Liouville fractional derivative. Adv. Difference Equ.,
2009, Article ID 981728, 1-47.

\bibitem{Bagley1} Bagley R L, Torvik P J. A theoretical basis for the
application of fractional calculus to viscoelasticity. J.
Rheology, 1983, 27, 201-210.

\bibitem{Bagley2} Bagley R L, Torvik P J. A different approach to the
analysis of viscoelastically damped structures. AIAA Journal,
1983, 21, 741-748.

\bibitem{Bagley3} Bagley R L, Torvik P J. On the appearance of the
fractional derivative in the behavior of real material. J. Appl.
Mechanics, 1983, 51, 294-298.

\bibitem{Cascaval} Cascaval R C, Eckstein E C, Frota C L, Godstein J A. Fractional telegraph equations. J. Math. Anal.
Appl., 2002, 276, 145-159.

\bibitem{F1} Furati K F, Tatar N E. An existence result for a
nonlocal fractional differential problem. J. Fract. Calc., 2004,
26, 43-51.

\bibitem{F2} Furati K F, Tatar N E. Behavior of solutions for a
weighted Cauchy-type fractional differential problem. J. Fract.
Calc., 2005, 28, 23-42.

\bibitem{F3} Furati K F, Kassim M D, Tatar N E. Existence and
uniqueness for a problem involving Hilfer fractional derivative.
Comput. Math. Appl., 2012, 64, 1616-1626.

\bibitem{F4} Furati K F, Kassim M D, Tatar N E. Non-existence of
global solutions for a differential equation involving Hilfer
fractional derivative. Electron. J. Diff. Equ., 2013, 2013, 1-10.

\bibitem{Hilfer} Hilfer R. Fractional time evolution, Applications of
fractional calculus in physics. World Scientific, New-Jersey,
London-Hong Kong, 2000, 87-130.

\bibitem{Ka1} Kassim M D, Furati K F, Tatar N E. On a differential
equation involving Hilfer-Hadamard fractional derivative. Abstr.
Appl. Anal., 2012, Article ID 391062, 1-17.

\bibitem{Ka2} Kassim M D, Tatar N E. Well-posedness and stability for
a differential problem with Hilfer-Hadamard fractional derivative.
Abstr. Appl. Anal.,2013, Article ID 605029, 1-12.

\bibitem{Kilbas} Kilbas A A, Srivastava H M, Trujillo J J. Theory
and Applications of Fractional Differential Equations, Elsevier
Science, 2006, 204.

\bibitem{K1} Kirane M, Medved M, Tatar N E. On the nonexistence of
blowing-up solutions to a fractional functional differential
equations. Georgian J. Math., 2012, 19, 127-144.

\bibitem{K2} Kirane M, Tatar N E. Nonexistence of solutions to a
hyperbolic equation with a time fractional damping. Z. Anal.
Anwendungen, 2006, 25, 131-142.

\bibitem{K3} Kirane M, Tatar N E. Absence of local and global
solutions to an elliptic system with time-fractional dynamical
boundary conditions. Siberian J. Math., 2007, 48, 477-488.

\bibitem{K4} Kirane M, Laskri Y, Tatar N E. Critical exponents of
Fujita type for certain evolution equations and systems with
spatio-temporal fractional derivatives. J. Math. Anal. Appl.,
2005, 312, 488-501.

\bibitem{V} Kiryakova V. Generalized Fractional Calculus and
Applications. John Wiley \& Sons Inc., 1994, New York.

\bibitem{R} Koeller R C. Application of fractional calculus to the theory
of viscoelasticity. J. Appl. Mechanics, 1984, 51, 299-307.

\bibitem{Laskri-Tatar(2010)} Laskri Y, Tatar N E. The critical
exponent for an ordinary fractional differential problem. Comput.
Math. Appl., 2010, 59, 1266-1270.

\bibitem{M1} Mainardi F, Gorenflo R. Time-fractional derivatives in
relaxation processes: a tutorial survey. Fract. Calc. Appl. Anal.,
2007, 10, 269--308.

\bibitem{M2} Mainardi F. Fractional Calculus and Waves in Linear
Viscoelasticity. Imperial College Press, 2010, London.

\bibitem{M} Messaoudi S A, Said-Houari B, Tatar N E. Global
existence and asymptotic behavior for a fractional differential
equation. Appl. Math. Comput., 2007, 188, 1955-1962.

\bibitem{Miller-Ross(1993)} Miller K S, Ross B. An Introduction to the
Fractional Calculus and Fractional Differential Equations. John
Wiley, 1993, New York.

\bibitem{Oldham-Spanier} Oldham K B, Spanier J. The Fractional
Calculus. Academic Press, 1974, New York, London.

\bibitem{Podl(1999)} Podlubny I. Fractional Differential Equations,
Mathematics in Sciences and Engineering. Academic Press, 1999,
San-Diego.

\bibitem{Podl(2002)} Podlubny I. Geometric and physical interpretation of
fractional integration and fractional differentiation. Fract.
Calcul. Anal. Appl., 2002, 5, 367-386.

\bibitem{P} Podlubny I, Petr\'{a}\v{s} I, Vinagre B M, O'Leary P,
Dor\v{c}\'{a}k L. Analogue realizations of fractional-order
controllers. Nonlinear Dynam., 2002, 29, 281-296.

\bibitem{Tatar} Tatar N E. Nonexistence results for a fractional problem
arising in thermal diffusion in fractal media. Chaos Solitons
Fractals, 2008, 36, 1205-1214.

\bibitem{T} Tatar N E. Existence results for an evolution problem with
fractional nonlocal conditions. Comput. Math. Appl., 2010, 60,
2971-2982.
\end{thebibliography}
\end{document}